\documentclass[11pt]{article}
\usepackage{amssymb,amsmath}
\usepackage[mathscr]{eucal}
\usepackage[cm]{fullpage}
\usepackage[english]{babel}
\usepackage[latin1]{inputenc}
\usepackage{xcolor}

\def\dom{\mathop{\mathrm{Dom}}\nolimits}
\def\im{\mathop{\mathrm{Im}}\nolimits}

\def\id{\mathrm{id}}
\def\N{\mathbb N}
\def\PT{\mathcal{PT}}
\def\T{\mathcal{T}}
\def\Sym{\mathcal{S}}

\def\A{\mathcal{A}}

\def\C{\mathcal{C}}

\def\POPI{\mathcal{POPI}}

\def\CI{\mathcal{CI}}
\def\OCI{\mathcal{OCI}}
\def\I{\mathcal{I}}

\def\ro{{\hspace{.2em}}\rho_R{\hspace{.2em}}}
\newtheorem{theorem}{Theorem}[section]
\newtheorem{proposition}[theorem]{Proposition}

\newtheorem{lemma}[theorem]{Lemma}

\newenvironment{proof}{\begin{trivlist}\item[\hskip%
\labelsep{\bf Proof.}]}%
{\qed\rm\end{trivlist}}
\newcommand{\qed}{{\unskip\nobreak
\hfil\penalty50\hskip .001pt \hbox{}
          \nobreak\hfil
         \vrule height 1.2ex width 1.1ex depth -.1ex
           \parfillskip=0pt\finalhyphendemerits=0\medbreak}}

\newcommand{\lastpage}{\addresss}

\newcommand{\addresss}{\small \sf  
\noindent{\sc V\'\i tor H. Fernandes}, 
Center for Mathematics and Applications (NovaMath) 
and Department of Mathematics, FCT NOVA, 
Faculdade de Ci\^encias e Tecnologia, 
Universidade Nova de Lisboa, 
Monte da Caparica, 
2829-516 Caparica, 
Portugal; 
e-mail: vhf@fct.unl.pt. 
}

\title{On the cyclic inverse monoid on a finite set}

\author{V\'\i tor H. Fernandes\footnote{This work is funded by national funds through the FCT - Funda\c c\~ao para a Ci\^encia e a Tecnologia, I.P., under the scope of the projects UIDB/00297/2020 and UIDP/00297/2020 (NovaMath - Center for Mathematics and Applications).}~
}

\begin{document}

\maketitle

\begin{abstract}
In this paper we study the cyclic inverse monoid $\CI_n$ on a set $\Omega_n$ with $n$ elements, 
i.e. the inverse submonoid of the symmetric inverse monoid on $\Omega_n$ consisting of all restrictions of the elements of a cyclic subgroup of order $n$ 
acting cyclically on $\Omega_n$. We show that $\CI_n$ has rank $2$ (for $n\geqslant2$) and $n2^n-n+1$ elements. 
Moreover, we give presentations of $\CI_n$ on $n+1$ generators and $\frac{1}{2}(n^2+3n+4)$ relations and on $2$ generators and $\frac{1}{2}(n^2-n+6)$ relations. 
We also consider the remarkable inverse submonoid $\OCI_n$ of $\CI_n$ constituted by all its order-preserving transformations. 
We show that $\OCI_n$ has rank $n$ and $3\cdot 2^n-2n-1$ elements. Furthermore, we exhibit presentations of $\OCI_n$ on $n+2$ generators and $\frac{1}{2}(n^2+3n+8)$ relations and on $n$ generators and $\frac{1}{2}(n^2+3n)$ relations.  
\end{abstract}

\medskip



\section*{Introduction}\label{presection} 

For $n\in\N$, let $\Omega_n$ be a set with $n$ elements, e.g. $\Omega_n=\{1,2,\ldots,n\}$.  
As usual, denote by $\PT_n$ the monoid (under composition) of all 
partial transformations on $\Omega_n$, 
by $\T_n$ the submonoid of $\PT_n$ of all full transformations on $\Omega_n$, 
by $\I_n$ the \textit{symmetric inverse monoid} on $\Omega_n$, i.e. 
the inverse submonoid of $\PT_n$ of all 
partial permutations on $\Omega_n$, 
and by $\Sym_n$ the \textit{symmetric group} on $\Omega$, 
i.e. the subgroup of $\PT_n$ of all 
permutations on $\Omega$. 

Let $G$ be a subgroup of $\Sym_n$ and define $\I_n(G)=\{\alpha\in\PT_n\mid \mbox{$\alpha=\sigma|_{\dom(\alpha)}$, for some $\sigma\in G$}\}$.  
It is easy to check that $\I_n(G)$ is an inverse submonoid of $\I_n$ whose group of units is precisely $G$. By taking $G=\Sym_n$, $G=\A_n$ or $G=\{\id_n\}$,
where $\A_n$ denotes the \textit{alternating group} on $\Omega_n$ and $\id_n$ is the identity transformation of $\Omega_n$, 
we obtain important and well-known inverse submonoids of $\I_n$. In fact, clearly,  
$\I_n(\Sym_n)=\I_n$ and $\I_n(\{\id_n\})=\mathscr{E}_n$, the semilattice of all idempotents of $\I_n$. On the other hand, 
$\I_n(\A_n)=\A^c_n$, the \textit{alternating semigroup} (see \cite[Chapters 6 and 10]{Lipscomb:1996}). 
In this work we are interested in studying the inverse monoid $\I_n(G)$ for an elementary but very important subgroup $G$ of $\Sym_n$, 
namely a cyclic subgroup of $\Sym_n$ of order $n$ acting cyclically on $\Omega_n$. 

\smallskip 

Recall that the \textit{rank} of a (finite) monoid $M$ is the minimum size of a generating set of $M$, i.e. 
the minimum of the set $\{|X|\mid \mbox{$X\subseteq M$ and $X$ generates $M$}\}$. 

For $n\geqslant3$ it is well-known that $\Sym_n$
has rank $2$ (as a semigroup, a monoid or a group) and 
$\T_n$, $\I_n$ and $\PT_n$ have ranks $3$, $3$ and $4$, respectively.
The survey \cite{Fernandes:2002survey} presents 
these results and similar ones for other classes of transformation monoids,
in particular, for monoids of order-preserving transformations and
for some of their extensions. 
For example, the rank of the extensively studied monoid of all order-preserving transformations of a $n$-chain is $n$,  
which was proved by Gomes and Howie \cite{Gomes&Howie:1992} in 1992. 
More recently, for instance, the papers 
\cite{
Araujo&al:2015,
Fernandes&al:2014,
Fernandes&al:2019,
Fernandes&Quinteiro:2014,
Fernandes&Sanwong:2014} 
are dedicated to the computation of the ranks of certain classes of transformation semigroups or monoids.

\smallskip 

A \textit{monoid presentation} is an ordered pair 
$\langle A\mid R\rangle$, where $A$ is a set, often called an \textit{alphabet}, 
and $R\subseteq A^*\times A^*$ is a set of relations of 
the free monoid $A^*$ generated by $A$. 
A monoid $M$ is said to be 
\textit{defined by a presentation} $\langle A\mid R\rangle$ if $M$ is
isomorphic to $A^*/\rho_R$, where $\rho_R$ denotes the smallest
congruence on $A^*$ containing $R$. 

A presentation for the symmetric group $\Sym_n$ was determined by Moore \cite{Moore:1897} over a century ago (1897). 
For the full transformation monoid $\T_n$, a presentation  
was given in 1958 by A\u{\i}zen\v{s}tat \cite{Aizenstat:1958} in terms of a certain 
type of two generators presentation for the symmetric group $\Sym_n$, 
plus an extra generator and seven more relations. 
Presentations for the partial transformation monoid $\PT_n$ 
and for the symmetric inverse monoid $\I_n$
were found by Popova \cite{Popova:1961} in 1961. 
In 1962, A\u{\i}zen\v{s}tat \cite{Aizenstat:1962} and Popova \cite{Popova:1962} exhibited presentations for the monoids of 
all order-preserving transformations and of all order-preserving partial transformations of a finite chain, respectively, and from the sixties until our days several authors obtained presentations for many classes of monoids. 
See also \cite{Ruskuc:1995}, the survey \cite{Fernandes:2002survey} and, 
for example, 
\cite{Cicalo&al:2015,
East:2011, 
Feng&al:2019,
Fernandes:2001, 
Fernandes&Gomes&Jesus:2004, 
Fernandes&Quinteiro:2016, 
Howie&Ruskuc:1995}. 

\smallskip 

Next, suppose that $\Omega_n$ is a chain, e.g. $\Omega_n=\{1<2<\cdots<n\}$.  
Given a partial transformation $\alpha\in\PT_n$ such that
$\dom(\alpha)=\{a_1<\cdots<a_t\}$, with $t\geqslant0$, we 
say that $\alpha$ is 
\textit{order-preserving} if $a_1\alpha\leqslant\cdots\leqslant a_t\alpha$
and that $\alpha$ is \textit{orientation-preserving} 
 if there exists no more than one index $i\in\{1,\ldots,t\}$ such that
$a_i\alpha>a_{i+1}\alpha$,
where $a_{t+1}$ denotes $a_1$. 
See \cite{Catarino&Higgins:1999,Fernandes:2001,Fernandes:2002survey,Higgins&Vernitski:2022,McAlister:1998}. 
We denote by $\POPI_n$ the submonoid of $\PT_n$ 
of all injective orientation-preserving
partial transformations. Notice that $\POPI_n$ is an inverse submonoid of $\I_n$ 
that was introduced and studied by the author in \cite{Fernandes:2000}. 
See also \cite{Fernandes&Gomes&Jesus:2004,Fernandes&Gomes&Jesus:2009}. 

\smallskip 

Now, consider the permutation 
$$
g=\begin{pmatrix} 
1&2&\cdots&n-1&n\\
2&3&\cdots&n&1
\end{pmatrix} 
$$
of $\Omega_n$ of order $n$ and denote by $\C_n$ the \textit{cyclic group} of order $n$ generated by $g$, i.e. 
$\C_n=\{1,g,g^2,\ldots,g^{n-1}\}$.  
Let us denote the monoid $\I_n(\C_n)$ by $\CI_n$.   
Then $\CI_n$ is an inverse submonoid of $\I_n$ whose group of units is $\C_n$. 
Moreover, as $\C_n\subseteq\POPI_n$ and 
each restriction of an orientation-preserving transformation is an  orientation-preserving transformation \cite[Proposition 2.1]{Fernandes:2000}, 
we also have that $\CI_n$ is an inverse submonoid of $\POPI_n$. 
Observe that $\CI_1=\I_1$ and $\CI_2=\I_2$. However $\CI_n\subsetneq\POPI_n$ for $n\geqslant3$. 
Given the definition of $\CI_n$, although it is not in general a \textit{monogenic} monoid, 
it seems appropriate to designate $\CI_n$ by the \textit{cyclic inverse monoid} on $\Omega_n$.  
A remarkable submonoid of $\CI_n$, which we denote by $\OCI_n$, is obtained when we consider all its order-preserving transformations. 
Clearly, $\OCI_n$ is an inverse submonoid of $\CI_n$. 

\smallskip 

This paper is organized as follows: in Section \ref{basics} we determine sizes and ranks of $\CI_n$ and $\OCI_n$; and 
in Section \ref{presentation} we give presentations of $\CI_n$ on $n+1$ generators and of $\OCI_n$ on $n+2$ generators 
followed by presentations of $\CI_n$ on $2$ generators  and of $\OCI_n$ on $n$ generators. 

\smallskip 

For general background on Semigroup Theory and standard notations, we refer to Howie's book \cite{Howie:1995}.

\smallskip 

We would like to point out that we made use of computational tools, namely GAP \cite{GAP4}.

\section{Sizes and ranks} \label{basics} 

We begin this section by calculating the size and the rank of $\CI_n$. 

Observe that 
$$
g^k=\begin{pmatrix} 
1&2&\cdots&n-k&n-k+1&\cdots&n\\
1+k&2+k&\cdots&n&1&\cdots&k
\end{pmatrix}, 
\quad\text{i.e.}\quad  
ig^k=\left\{\begin{array}{ll}
i+k & \mbox{if $1\leqslant i\leqslant n-k$}\\
i+k-n & \mbox{if $n-k+1\leqslant i\leqslant n$},
\end{array}\right.
$$
for $0\leqslant k\leqslant n-1$. Hence, for each pair $1\leqslant i,j\leqslant n$, there exists a unique $k\in\{0,1,\ldots,n-1\}$ such that $ig^k=j$. 
In fact, for $1\leqslant i,j\leqslant n$ and $k\in\{0,1,\ldots,n-1\}$, 
it is easy to show that: 
\begin{description}
\item if $i\leqslant j$ then $ig^k=j$ if and only if $k=j-i$; 

\item if $i>j$ then $ig^k=j$ if and only if $k=n+j-i$. 
\end{description} 

Thus, we can immediately conclude the following property of $\CI_n$:

\begin{lemma}\label{unique}
Any nonempty transformation of $\CI_n$ has exactly one extension in $\C_n$. 
\end{lemma}

It follows that the number of nonempty elements of $\CI_n$ coincides with the number of distinct nonempty restrictions of elements of $\C_n$, i.e. 
$|\CI_n\setminus\{\emptyset\}|=n\sum_{\ell=1}^{n}\binom{n}{\ell}=n(-1+\sum_{\ell=0}^{n}\binom{n}{\ell})=n(2^n-1)$. 

Therefore, we have:  

\begin{theorem}
For $n\geqslant1$, $|\CI_n|=n2^n-n+1$.
\end{theorem}

For $X\subseteq\Omega_n$, denote by $\id_X$ the partial identity with domain $X$, i.e. $\id_X=\id_n|_X$. 
Let 
$$
e_i=\id_{\Omega_n\setminus\{i\}}=
\begin{pmatrix} 1&\cdots&i-1&i+1&\cdots&n\\
1&\cdots&i-1&i+1&\cdots&n
\end{pmatrix}\in\CI_n, 
$$
for $i=1,2,\ldots,n$. Clearly, for $1\leqslant i,j\leqslant n$, we have $e_i^2=e_i$ and $e_ie_j=\id_{\Omega_n\setminus\{i,j\}}=e_je_i$. 
More generally, for any $X\subseteq\Omega_n$, we get $\Pi_{i\in X}e_i=\id_{\Omega_n\setminus X}$. 

Now, take $\alpha\in\CI_n$. Then, by definition, $\alpha=g^i|_{\dom(\alpha)}$, 
for some $i\in\{0,1,\ldots,n-1\}$, and so we obtain $\alpha=\id_{\dom(\alpha)}g^i=(\Pi_{k\in\Omega_n\setminus\dom(\alpha)}e_k)g^i$. 
Hence 
$
\{g,e_1,e_2,\ldots,e_n\}
$
is a generating set of $\CI_n$. Since $e_i=g^{n-i+1}e_1g^{i-1}$ for all $i\in\{1,2,\ldots,n\}$, 
it follows that $\{g,e_1\}$ is also a generating set of $\CI_n$. 
For $n\geqslant2$,  as $|\C_n|>1$ and $\C_n$ is the group of units of $\CI_n$,
the monoid $\CI_n$ cannot be generated by less than two elements. So, we have: 

\begin{theorem}
For $n\geqslant2$, the monoid $\CI_n$ has rank $2$.
\end{theorem}

Observe that, as a monoid, $\CI_1=\I_1$ has rank $1$. However, as a semigroup, $\CI_n$ has rank $2$ for all $n\in\N$. 

\medskip

Next, we deduce the size and rank of $\OCI_n$. 

Clearly, the elements of $\OCI_n$ are all restrictions of 
$$
g^ke_1\cdots e_k=\begin{pmatrix} 
1&2&\cdots&n-k\\
1+k&2+k&\cdots&n
\end{pmatrix}
\quad\text{and}\quad 
g^ke_{k+1}\cdots e_n=\begin{pmatrix} 
n-k+1&n-k+2&\cdots&n\\
1&2&\cdots&k
\end{pmatrix}, 
$$
for $0\leqslant k\leqslant n-1$, whence 
$$
|\OCI_n|=1 + \sum_{k=0}^{n-1}\left(\sum_{i=1}^{n-k} \binom{n-k}{i} + \sum_{i=1}^{k}\binom{k}{i}\right) = 
\sum_{k=0}^{n-1}\left( 2^{n-k}-1+2^k-1\right) = 3\cdot 2^n-2n-2. 
$$

Thus, we have:  

\begin{theorem}
For $n\geqslant1$, $|\OCI_n|=3\cdot 2^n-2n-2$.
\end{theorem}

Now, let 
$$
x=ge_1=\begin{pmatrix} 
1&2&\cdots&n-1\\
2&3&\cdots&n
\end{pmatrix}
\quad\text{and}\quad 
y=x^{-1}=g^{n-1}e_n=\begin{pmatrix} 
2&3&\cdots&n\\
1&2&\cdots&n-1
\end{pmatrix}. 
$$
Then, it is easy to check that 
\begin{equation}\label{xkyk}
x^k=g^ke_1\cdots e_k
\quad\text{and}\quad 
y^k= g^{n-k}e_{n-k+1}\cdots e_n,  
\end{equation}
for $1\leqslant k\leqslant n$. 
Hence, the elements of $\OCI_n$ are all restrictions of $\id_n$ and of $x^k$ and $y^k$, with $1\leqslant k\leqslant n-1$. 
It follows that $\{x,y,e_1,e_2,\ldots,e_n\}$ generates the monoid $\OCI_n$. Since $xy=e_n$ and $yx=e_1$, we have that 
$\{x,y,e_2,\ldots,e_{n-1}\}$ is also a generating set of $\OCI_n$. On the other hand, since the group of units of $\OCI_n$ is trivial 
(the identity is the only order-preserving permutation), then any set of generators of $\OCI_n$ must contain at least one element of each possible image of size $n-1$. 
As we have elements of $\OCI_n$ with all $n$ possible distinct images of size $n-1$ 
(for instance the partial identities $e_1,\ldots,e_n$), it follows that any set of generators of $\OCI_n$ must contain at least $n$ elements. 
Therefore, we conclude that: 

\begin{theorem}
For $n\geqslant1$, the monoid $\OCI_n$ has rank $n$.
\end{theorem}

Notice that $\OCI_1=\CI_1=\I_1$. 

\section{Presentations}\label{presentation}

In this section, we aim to determine presentations for $\CI_n$ and $\OCI_n$. 

We begin by determining a presentation of $\CI_n$ on $n+1$ generators and then, 
by applying applying \textit{Tietze transformations}, we deduce a presentation for $\CI_n$ on $2$ generators. 

\smallskip

At this point, we recall some basic notions and results related to the concept of a monoid presentation.

\smallskip 

Let $A$ be an alphabet and consider the free monoid $A^*$ generated by $A$. 
The elements of $A$ and of $A^*$ are called \textit{letters} and \textit{words}, respectively. 
The empty word is denoted by $1$. 
A pair $(u,v)$ of $A^*\times A^*$ is called a
\textit{relation} of $A^*$ and it is usually represented by $u=v$. 
A relation $u=v$ of $A^*$ is said to be a \textit{consequence} of $R$ if $u{\hspace{.11em}}\rho_R{\hspace{.11em}}v$. 
Let $X$ be a generating set of $M$ and let $\phi: A\longrightarrow M$ be an injective mapping 
such that $A\phi=X$. 
Let $\varphi: A^*\longrightarrow M$ be the (surjective) homomorphism of monoids that extends $\phi$ to $A^*$. 
We say that $X$ satisfies (via $\varphi$) a relation $u=v$ of $A^*$ if $u\varphi=v\varphi$. 
For more details see
\cite{Lallement:1979} or \cite{Ruskuc:1995}. 
A direct method to find a presentation for a monoid
is described by the following well-known result (e.g.  see \cite[Proposition 1.2.3]{Ruskuc:1995}).  

\begin{proposition}\label{provingpresentation} 
Let $M$ be a monoid generated by a set $X$, let $A$ be an alphabet 
and let $\phi: A\longrightarrow M$ be an injective mapping 
such that $A\phi=X$. 
Let $\varphi:A^*\longrightarrow M$ be the (surjective) homomorphism 
that extends $\phi$ to $A^*$ and let $R\subseteq A^*\times A^*$.
Then $\langle A\mid R\rangle$ is a presentation for $M$ if and only
if the following two conditions are satisfied:
\begin{enumerate}
\item
The generating set $X$ of $M$ satisfies (via $\varphi$) all the relations from $R$;  
\item 
If $u,v\in A^*$ are any two words such that 
the generating set $X$ of $M$ satisfies (via $\varphi$) the relation $u=v$ then $u=v$ is a consequence of $R$.
\end{enumerate} 
\end{proposition}

\smallskip

Given a presentation for a monoid, another method to find a new
presentation consists in applying Tietze transformations. For a
monoid presentation $\langle A\mid R\rangle$, the 
four \textit{elementary Tietze transformations} are:

\begin{description}
\item(T1)
Adding a new relation $u=v$ to $\langle A\mid R\rangle$,
provided that $u=v$ is a consequence of $R$;
\item(T2)
Deleting a relation $u=v$ from $\langle A\mid R\rangle$,
provided that $u=v$ is a consequence of $R\backslash\{u=v\}$;
\item(T3)
Adding a new generating symbol $b$ and a new relation $b=w$, where
$w\in A^*$;
\item(T4)
If $\langle A\mid R\rangle$ possesses a relation of the form
$b=w$, where $b\in A$, and $w\in(A\backslash\{b\})^*$, then
deleting $b$ from the list of generating symbols, deleting the
relation $b=w$, and replacing all remaining appearances of $b$ by
$w$.
\end{description}

The next result is well-known (e.g. see \cite{Ruskuc:1995}): 

\begin{proposition} \label{tietze}
Two finite presentations define the same monoid if and only if one
can be obtained from the other by a finite number of elementary
Tietze transformations $(T1)$, $(T2)$, $(T3)$ and $(T4)$.  
\end{proposition}

\medskip 

Now, consider the alphabet $A=\{g,e_1,e_2,\ldots,e_n\}$ and the set $R$ formed by the following monoid relations: 
\begin{description}
\item $(R_1)$ $g^n=1$; 

\item $(R_2)$ $e_i^2=e_i$, for $1\leqslant i\leqslant n$;

\item $(R_3)$ $e_ie_j=e_je_i$, for $1\leqslant i<j\leqslant n$; 

\item $(R_4)$ $ge_1=e_ng$ and $ge_{i+1}=e_ig$, for $1\leqslant i\leqslant n-1$; 

\item $(R_5)$ $ge_1e_2\cdots e_n=e_1e_2\cdots e_n$.
\end{description}

Observe that $|R|=\frac{1}{2}(n^2+3n+4)$. 

\smallskip 

We aim to show that the monoid $\CI_n$ is defined by the presentation $\langle A \mid R\rangle$. 

\smallskip 

Let $\phi:A\longrightarrow \CI_n$ be the mapping defined by
$$
g\phi=g ,\quad e_i\phi=e_i, \mbox{~for $1\leqslant i\leqslant n$}, 
$$
and let $\varphi:A^*\longrightarrow \CI_n$ be the homomorphism of monoids that extends $\phi$ to $A^*$. 
Notice that we are using the same symbols for the letters of the alphabet $A$ and for the generating set of $\CI_n$,
which simplifies notation and, within the context, will not cause ambiguity. 

\smallskip 

It is a routine matter to check the following lemma. 

\begin{lemma}\label{genrel}
The set of generators $\{g,e_1,e_2,\ldots,e_n\}$ of $\CI_n$ satisfies (via $\varphi$) all the relations from $R$.
\end{lemma}

This lemma assures us that, if $u,v\in A^*$ are such that the relation $u=v$ is a consequence of $R$, 
then $u\varphi=v\varphi$. 

\smallskip 

Next, in order to prove that any relation satisfied by the generating set $\{g,e_1,e_2,\ldots,e_n\}$ of $\CI_n$ is a consequence of $R$, 
we first present two lemmas whose proofs are routine. 

\begin{lemma}\label{pre1} 
Let $u\in A^*$. Then, there exist $m\in\{0,1,\ldots,n-1\}$, $1\leqslant i_1 < \cdots < i_k\leqslant n$ and $0\leqslant k\leqslant n$ such that 
the relation $u=g^m e_{i_1}\cdots e_{i_k}$ is a consequence of relations $R_1$ to $R_4$. 
\end{lemma}

\begin{lemma}\label{pre2} 
For all $m\in\N$, 
the relation $g^me_1e_2\cdots e_n=e_1e_2\cdots e_n$
is a consequence of $R_5$. 
\end{lemma}

Now, we may prove the following result.

\begin{theorem}\label{firstpres} 
The monoid $\CI_n$ is defined by the presentation $\langle A \mid R\rangle$ on $n+1$ generators and $\frac{1}{2}(n^2+3n+4)$ relations. 
\end{theorem} 
\begin{proof} Taking into account Proposition \ref{provingpresentation} and Lemma \ref{genrel}, 
it remains to prove that any relation satisfied by the generating set $\{g,e_1,e_2,\ldots,e_n\}$ of $\CI_n$ is a consequence of $R$. 

Let $u,v\in A^*$ be such that $u\varphi=v\varphi$. We aim to show that $u\ro v$. 

By Lemma \ref{pre1}, there exist $m\in\{0,1,\ldots,n-1\}$, $1\leqslant i_1 < \cdots < i_k\leqslant n$ and $0\leqslant k\leqslant n$ 
such that $u\ro g^m e_{i_1}\cdots e_{i_k}$ and 
$m'\in\{0,1,\ldots,n-1\}$, $1\leqslant i'_1 < \cdots < i'_{k'}\leqslant n$ and $0\leqslant k'\leqslant n$ 
such that $v\ro g^{m'} e_{i'_1}\cdots e_{i'_k}$. 

Take $\alpha=u\varphi$. 
Since $\alpha=g^m e_{i_1}\cdots e_{i_k}$, it follows that $\im(\alpha)=\Omega_n\setminus\{i_1,\ldots,i_k\}$ and 
$\alpha=g^m|_{\dom(\alpha)}$. Similarly, as also $\alpha=v\varphi$, 
from $\alpha= g^{m'} e_{i'_1}\cdots e_{i'_k}$, 
we get $\im(\alpha)=\Omega_n\setminus\{i'_1,\ldots,i'_{k'}\}$ and $\alpha= g^{m'}|_{\dom(\alpha)}$. 
Hence $k'=k$ and $\{i'_1,\ldots,i'_k\}=\{i_1,\ldots,i_k\}$. 

If $\alpha\neq\emptyset$ then, by Lemma \ref{unique}, $m=m'$ and so $u\ro v$.  
On the other hand, 
if $\alpha=\emptyset$, i.e. $k=n$, then 
$u\ro g^m e_1e_2\cdots e_n \ro e_1e_2\cdots e_n \ro g^{m'} e_1e_2\cdots e_n \ro v$, 
by Lemma \ref{pre2}, as required. 
\end{proof}

Next, by using Tietze transformations and applying Proposition \ref{tietze}, we deduce from the previous presentation for $\CI_n$ 
a new one on the $2$-generators set $\{g,e_1\}$ of $\CI_n$.  

\smallskip 

Recall that $e_i=g^{n-i+1}e_1g^{i-1}$ for all $i\in\{1,2,\ldots,n\}$. 

\smallskip 

We will proceed as follows: first, by applying T1, we add the relations $e_i=g^{n-i+1}e_1g^{i-1}$, for $2\leqslant i\leqslant n$; 
secondly, we apply T4 to each of the relations $e_i=g^{n-i+1}e_1g^{i-1}$ with $i\in\{2,3,\ldots,n\}$; 
finally, by using the relation $R_1$, we simplify the new relations obtained, eliminating the trivial ones or those that are deduced from others. 
In what follows, we perform this procedure for each of the sets of relations $R_1$ to $R_5$. 
\begin{description}
\item $(R_1)$ 
There is nothing to do for this relation. 

\item $(R_2)$ 
For $2\leqslant i\leqslant n$, from $e_i^2=e_i$, we have 
$$
g^{n-i+1}e_1g^{i-1} g^{n-i+1}e_1g^{i-1} = g^{n-i+1}e_1g^{i-1},
$$
which is equivalent to $e_1^2=e_1$. 

\item $(R_3)$ 
For $1\leqslant i<j\leqslant n$, from $e_ie_j=e_je_i$, we get 
$$
g^{n-i+1}e_1g^{i-1} g^{n-j+1}e_1g^{j-1} = g^{n-j+1}e_1g^{j-1} g^{n-i+1}e_1g^{i-1} 
$$
and this relation is equivalent to $e_1g^{n-j+i}e_1g^{n-i+j} = g^{n-j+i}e_1g^{n-i+j}e_1$. 

\item $(R_4)$  
From $ge_1=e_ng$, we obtain 
$$
ge_1=ge_1g^{n-1}g,
$$
which is equivalent to $e_1=e_1$. 
On the other hand, for $1\leqslant i\leqslant n-1$, 
from $ge_{i+1}=e_ig$ we get 
$$
g g^{n-i}e_1g^{i}= g^{n-i+1}e_1g^{i-1} g 
$$
and this relation is equivalent to $e_1=e_1$. 

\item $(R_5)$ 
Finally, from $ge_1e_2\cdots e_n=e_1e_2\cdots e_n$ we get 
$$
g e_1 (g^{n-1}e_1g)(g^{n-2}e_1g^{2})\cdots(ge_1g^{n-1}) = e_1 (g^{n-1}e_1g)(g^{n-2}e_1g^{2})\cdots(ge_1g^{n-1}), 
$$
i.e. the relation $g(e_1g^{n-1})^n=(e_1g^{n-1})^n$. 
\end{description}

\smallskip 

Therefore, let us consider the following set $Q$ of monoid relations on the alphabet $B=\{g,e\}$: 
\begin{description}
\item $(Q_1)$ $g^n=1$; 

\item $(Q_2)$ $e^2=e$; 

\item $(Q_3)$ $eg^{n-j+i}eg^{n-i+j} = g^{n-j+i}eg^{n-i+j}e$, for $1\leqslant i<j\leqslant n$; 

\item $(Q_4)$ $g(eg^{n-1})^n=(eg^{n-1})^n$.  

\end{description}
Notice that 
$|Q|=\frac{1}{2}(n^2-n+6)$. 

Thus, by considering the mapping $B\longrightarrow\CI_n$ defined by $g\longmapsto g$ and $e\longmapsto e_1$, 
we have: 

\begin{theorem}\label{rankpres} 
The monoid $\CI_n$ is defined by the presentation $\langle B \mid Q\rangle$ on $2$ generators and $\frac{1}{2}(n^2-n+6)$ relations. 
\end{theorem} 

\medskip

Now, we focus our attention on the monoid $\OCI_n$. 

\smallskip 

Consider the alphabet $C=\{x,y,e_1,e_2,\ldots,e_n\}$ and the set $U$ formed by the following monoid relations: 
\begin{description}

\item $(U_1)$ $e_i^2=e_i$, for $1\leqslant i\leqslant n$;

\item $(U_2)$ $xy=e_n$ and $yx=e_1$;

\item $(U_3)$ $xe_1=x$ and $e_1y=y$; 

\item $(U_4)$ $e_ie_j=e_je_i$, for $1\leqslant i<j\leqslant n$; 

\item $(U_5)$ $xe_{i+1}=e_ix$, for $1\leqslant i\leqslant n-1$; 

\item $(U_6)$ $xe_2\cdots e_n=e_1e_2\cdots e_n$.
\end{description}

Observe that $|U|=\frac{1}{2}(n^2+3n+8)$. 

\smallskip 

Below, we show that the monoid $\OCI_n$ is defined by the presentation $\langle C \mid U\rangle$. 

\smallskip 

Let $\theta:C\longrightarrow \OCI_n$ be the mapping defined by
$$
x\theta=x ,\quad y\theta=y, \quad e_i\theta=e_i, \mbox{~for $1\leqslant i\leqslant n$}, 
$$
and let $\vartheta:C^*\longrightarrow \OCI_n$ be the homomorphism of monoids that extends $\theta$ to $C^*$. 

\smallskip 

It is a routine matter to check:  

\begin{lemma}\label{genreloci}
The set of generators $\{x,y,e_1,e_2,\ldots,e_n\}$ of $\OCI_n$ satisfies (via $\vartheta$) all the relations from $U$.
\end{lemma}

As a consequence of the previous lemma, if $u,v\in C^*$ are such that the relation $u=v$ is a consequence of $U$, 
then $u\vartheta=v\vartheta$. 

\smallskip 

Next, in order to prove that any relation satisfied by the generating set $\{x,y,e_1,e_2,\ldots,e_n\}$ of $\CI_n$ is a consequence of $U$, 
we first present a series of lemmas. 

\begin{lemma}\label{u3}
The relations $e_nx=x$ and $ye_n=y$ are consequences of $U_2$ and $U_3$.
\end{lemma}
\begin{proof}
Denote the congruence $\rho_{U_2\cup U_3}$ on $C^*$ by $\approx$. 
Then $e_nx\approx (xy)x\approx x(yx)\approx xe_1\approx x$
and, similarly, 
$ye_n\approx y(xy)\approx (yx)y\approx e_1y\approx y$, as required. 
\end{proof} 

\begin{lemma}\label{u5}
The relations $e_{i+1}y=ye_i$, for $1\leqslant i\leqslant n-1$, are consequences of $U_2$ to $U_5$. 
\end{lemma}
\begin{proof}
Let us denote the congruence $\rho_{U_2\cup U_3\cup U_4\cup U_5}$ on $C^*$ by $\approx$. Let $1\leqslant i\leqslant n-1$. 
Then 
$$
e_{i+1}y \approx e_{i+1}e_1y \approx e_1e_{i+1}y \approx yxe_{i+1}y \approx ye_ixy \approx ye_ie_n \approx ye_ne_i\approx ye_i,   
$$
as required. 
\end{proof} 

\begin{lemma}\label{powers}
The relations $x^je_i=x^j=e_{n-i+1}x^j$ and $e_iy^j=y^j=y^je_{n-i+1}$, for $1\leqslant i\leqslant j\leqslant n$, are consequences of $U_2$ to $U_5$. 
\end{lemma}
\begin{proof}
Denote the congruence $\rho_{U_2\cup U_3\cup U_4\cup U_5}$ on $C^*$ by $\approx$. 

First, we show by induction on $i$ that  
\begin{equation}\label{powersfix}
\mbox{$x^ie_i\approx x^i\approx e_{n-i+1}x^i$ and $e_iy^i\approx y^i\approx y^ie_{n-i+1}$, for $1\leqslant i\leqslant n$.}
\end{equation}
If $i=1$ then $xe_1\approx x\approx e_1x$ and $e_1y\approx y\approx ye_n$, by $U_3$ and Lemma \ref{u3}. 

Now, suppose that $x^ie_i\approx x^i\approx e_{n-i+1}x^i$ and $e_iy^i\approx y^i\approx y^ie_{n-i+1}$, for some $1\leqslant i\leqslant n-1$. 
Hence 
$$
x^{i+1}e_{i+1} \approx x^ixe_{i+1} \approx x^ie_ix \approx x^ix\approx x^{i+1} 
\approx xx^i \approx xe_{n-i+1}x^i \approx e_{n-i}xx^i \approx e_{n-i}x^{i+1} 
$$
and 
$$
e_{i+1}y^{i+1} \approx e_{i+1}yy^i  \approx ye_iy^i \approx yy^i \approx y^{i+1}
\approx y^i y\approx y^ie_{n-i+1}y \approx y^iye_{n-i} \approx y^{i+1}e_{n-i}. 
$$ 
Thus, we have proved (\ref{powersfix}).  

Next, let $1\leqslant i\leqslant j\leqslant n$. Then 
$$
x^je_i\approx x^{j-i}x^ie_i \approx  x^{j-i}x^i \approx x^j \approx x^i x^{j-i} \approx e_{n-i+1}x^ix^{j-i} \approx e_{n-i+1}x^j
$$ 
and 
$$
e_iy^j \approx e_iy^iy^{j-i} \approx y^iy^{j-i} \approx y^j \approx y^{j-i}y^i \approx y^{j-i}y^ie_{n-i+1} \approx y^je_{n-i+1}, 
$$
as required. 
\end{proof} 

From now on, denote the congruence $\rho_U$ on $C^*$ by $\approx$. 

\begin{lemma}\label{zero}
The relations 
$e_1\cdots e_n x =  x e_1\cdots e_n  = e_1\cdots e_n = e_1\cdots e_n y = y e_1\cdots e_n$ are consequences of $U$. 
\end{lemma}
\begin{proof}
First, we have $x e_1\cdots e_n  \approx x e_2\cdots e_n  \approx e_1\cdots e_n$, by $U_3$ and $U_6$. 
Secondly, 
$$
e_1\cdots e_n  x\approx e_1\cdots e_{n-1}  x\approx x e_2\cdots e_n \approx x e_1\cdots e_n \approx e_1\cdots e_n, 
$$
by Lemma \ref{u3}, $U_5$, $U_3$ and the first relation we proved.  
On the other hand, 
$$
ye_1\cdots e_n \approx yxe_1\cdots e_n \approx e_1e_1\cdots e_n \approx e_1\cdots e_n 
$$
by the first relation we proved, $U_2$ and $U_1$. Finally, 
$$
e_1\cdots e_n y \approx e_1ye_1\cdots e_{n-1}  \approx y e_1\cdots e_{n-1} \approx y e_n e_1\cdots e_{n-1} \approx y e_1\cdots e_n \approx e_1\cdots e_n,  
$$ 
by Lemma \ref{u5}, $U_3$, Lemma \ref{u3}, $U_4$ and the third relation we proved, 
as required. 
\end{proof}

\begin{lemma}\label{nzero}
The relations $x^n=e_1\cdots e_n =y^n$ are consequences of $U$. 
\end{lemma}
\begin{proof}
By Lemmas \ref{powers} and \ref{zero}, for $z\in\{x,y\}$, we have 
$$
z^n\approx z^ne_1\cdots e_n \approx  z^{n-1}e_1\cdots e_n \approx \cdots \approx z^2e_1\cdots e_n \approx ze_1\cdots e_n \approx e_1\cdots e_n, 
$$
as required. 
\end{proof} 

\begin{lemma}\label{commuting}
Let $z\in\{x,y\}$ and $u\in\{e_1,\ldots, e_n\}^*$. 
Then, there exists $v\in\{e_1,\ldots, e_n\}^*$ such that $uz=zv$ is a consequence of $U$. 
\end{lemma}
\begin{proof}
By applying $U_4$ and $U_1$, we obtain $u\approx e_{i_1}\cdots e_{i_k}$, for some $1\leqslant i_1<\cdots<i_k\leqslant n$ and $0\leqslant k\leqslant n$. 

Suppose that $z=x$. If $i_k<n$ then $ux\approx e_{i_1}\cdots e_{i_k}x \approx x e_{i_1+1}\cdots e_{i_k+1}$, by $U_5$. 
On the other hand, if $i_k=n$ then $ux\approx e_{i_1}\cdots e_{i_{k-1}}e_nx \approx e_{i_1}\cdots e_{i_{k-1}}x \approx x e_{i_1+1}\cdots e_{i_{k-1}+1}$, by Lemma \ref{u3} and $U_5$. 

Suppose that $z=y$. If $i_1>1$ then $uy\approx e_{i_1}\cdots e_{i_k}y\approx y e_{i_1-1}\cdots e_{i_k-1}$, by Lemma \ref{u5}. 
On the other hand, if $i_1=1$ then $uy\approx  e_{i_1}\cdots e_{i_k}y\approx e_1y e_{i_2-1}\cdots e_{i_k-1}\approx  y e_{i_2-1}\cdots e_{i_k-1}$, 
by Lemma \ref{u5} and $U_3$. 
\end{proof} 

\begin{lemma}\label{leftxy}
Let $w\in C^*$. 
Then, there exist $z\in\{x,y\}$, $u\in\{e_1,\ldots, e_n\}^*$ and $0\leqslant r\leqslant n-1$ such that $w=z^ru$ is a consequence of $U$. 
\end{lemma}
\begin{proof} 
We proceed by induction on $|w|$. 

If $|w|\leqslant1$ then there is nothing to prove. 

So, let us admit that the lemma is valid for any word $w\in C^*$ such that $|w|=m\geqslant1$. 

Take $w\in C^*$ such that $|w|=m+1$ and let $w_1\in C^*$ and $a\in C$ be such that $w=w_1a$. 
By the induction hypothesis there exist $z\in\{x,y\}$, $u_1\in\{e_1,\ldots, e_n\}^*$ and $0\leqslant r\leqslant n-1$ such that $w_1\approx z^ru_1$. 

If $a\in\{e_1,\ldots,e_n\}$ then $w\approx z^ru_1a$ and $u_1a\in \{e_1,\ldots, e_n\}^*$ and so, in this case, the lemma is proved. 

On the other hand, suppose that $a\in\{x,y\}$. Then, by Lemma \ref{commuting}, $u_1a\approx av_1$, for some $v_1\in\{e_1,\ldots, e_n\}^*$. 

If $r=0$ then $w\approx u_1a\approx av_1$ and so, also in this case, the lemma is proved. 

Therefore, suppose that $r\geqslant1$. 

If $a=z$ then $w\approx z^ru_1z\approx z^{r+1}v_1$. In this case, if $r\leqslant n-2$ then the lemma is proved. 
On the other hand, if $r=n-1$ then $w\approx z^nv_1\approx e_1\cdots e_nv_1(\approx e_1\cdots e_n)$, by Lemma \ref{nzero}, which proves the lemma also in this case.   

Finally, suppose that $a\neq z$. Then, we have $w\approx z^ru_1a\approx z^rav_1\approx z^{r-1}e_iv_1$, with $i=1$ if $z=y$ and $i=n$ if $z=x$, by $U_2$. 
So, in this case too, the lemma is proved. 
\end{proof} 

We are now in a position to prove that: 

\begin{theorem}\label{firstpresoci} 
The monoid $\OCI_n$ is defined by the presentation $\langle C \mid U\rangle$ on $n+2$ generators and $\frac{1}{2}(n^2+3n+8)$ relations. 
\end{theorem} 
\begin{proof}
In view of Proposition \ref{provingpresentation} and Lemma \ref{genreloci}, 
it remains to prove that any relation satisfied by the generating set $\{x,y,e_1,e_2,\ldots,e_n\}$ of $\OCI_n$ is a consequence of $U$. 

Let $w_1,w_2\in C^*$ be such that $w_1\vartheta=w_2\vartheta$. We aim to show that $w_1\approx w_2$. 

By Lemma \ref{leftxy} there exist $z_1,z_2\in\{x,y\}$, $u_1,u_2\in\{e_1,\ldots, e_n\}^*$ and $0\leqslant r_1,r_2\leqslant n-1$ 
such that $w_1\approx z_1^{r_1}u_1$ and $w_2\approx z_2^{r_2}u_2$. 
By applying $U_4$ and $U_1$ to $u_1$ and $u_2$, we may find $1\leqslant i_1<\cdots<i_{k_1}\leqslant n$ and $1\leqslant j_1<\cdots<j_{k_2}\leqslant n$, 
with $0\leqslant k_1,k_2\leqslant n$, such that 
$w_1\approx z_1^{r_1}e_{i_1}\cdots e_{i_{k_1}}$ and $w_2\approx z_2^{r_2}e_{j_1}\cdots e_{j_{k_2}}$. 

\smallskip 

First, let us suppose that $z_1=z_2=x$. 

Let $0\leqslant t_1\leqslant k_1$ and  $0\leqslant t_2\leqslant k_2$ be such that 
$i_{t_1}\leqslant r_1 <i_{t_1+1}$ and $j_{t_2}\leqslant r_2 <j_{t_2+1}$ (where $i_{k_1+1}=j_{k_2+1}=n$). 

Since $x^{r_1}\approx x^{r_1}e_1\cdots e_{r_1}$, by Lemma \ref{powers}, then we have 
$$
w_1\approx x^{r_1}e_{i_1}\cdots e_{i_{k_1}}\approx x^{r_1}e_1\cdots e_{r_1}e_{i_1}\cdots e_{i_{k_1}}\approx 
x^{r_1}e_1\cdots e_{r_1}e_{i_{t_1+1}}\cdots e_{i_{k_1}}. 
$$
Similarly, we obtain $w_2\approx x^{r_2}e_1\cdots e_{r_2}e_{j_{t_2+1}}\cdots e_{i_{k_2}}$. 

On the other hand, in view of (\ref{xkyk}), $w_1\vartheta=x^{r_1}e_{i_1}\cdots e_{i_{k_1}}=g^{r_1}e_1\cdots e_{r_1}e_{i_1}\cdots e_{i_{k_1}}
=g^{r_1}e_1\cdots e_{r_1}e_{i_{t_1+1}}\cdots e_{i_{k_1}}$ and, similarly, 
$w_2\vartheta=g^{r_2}e_1\cdots e_{r_2}e_{j_{t_2+1}}\cdots e_{j_{k_2}}$. 
Hence, we have 
$$
w_1\vartheta=g^{r_1}|_{\dom(w_1\vartheta)}
\quad\text{and}\quad 
\im(w_1\vartheta)=\Omega_n\setminus\{1,\ldots,r_1, i_{t_1+1}, \ldots, i_{k_1}\}
$$ 
and
$$
w_2\vartheta=g^{r_2}|_{\dom(w_2\vartheta)}
\quad\text{and}\quad 
\im(w_2\vartheta)=\Omega_n\setminus\{1,\ldots,r_2, j_{t_2+1}, \ldots, j_{k_2}\}. 
$$ 
Since $w_1\vartheta=w_2\vartheta$, in particular we have $\im(w_1\vartheta)=\im(w_2\vartheta)$ and so 
$$
\{1,\ldots,r_1, i_{t_1+1}, \ldots, i_{k_1}\}=\{1,\ldots,r_2, j_{t_2+1}, \ldots, j_{k_2}\}. 
$$

If $w_1\vartheta=\emptyset$ then $\im(w_1\vartheta)=\emptyset=\im(w_2\vartheta)$, whence 
$$
\{1,\ldots,r_1, i_{t_1+1}, \ldots, i_{k_1}\}=\Omega_n=\{1,\ldots,r_2, j_{t_2+1}, \ldots, j_{k_2}\}
$$
and so, by Lemma \ref{zero}, we have 
$$
w_1\approx x^{r_1}e_1\cdots e_{r_1}e_{i_{t_1+1}}\cdots e_{i_{k_1}} =  x^{r_1}e_1\cdots e_n \approx e_1\cdots e_n 
\approx x^{r_2}e_1\cdots e_n = x^{r_2}e_1\cdots e_{r_2}e_{j_{t_2+1}}\cdots e_{i_{k_2}} \approx  w_2. 
$$

On the other hand, if $w_1\vartheta\neq\emptyset$, from $g^{r_1}|_{\dom(w_1\vartheta)}=w_1\vartheta= w_2\vartheta=g^{r_2}|_{\dom(w_2\vartheta)}$, 
we have $r_1=r_2$, by Lemma \ref{unique}, and so 
$$
w_1\approx x^{r_1}e_1\cdots e_{r_1}e_{i_{t_1+1}}\cdots e_{i_{k_1}} = x^{r_2}e_1\cdots e_{r_2}e_{j_{t_2+1}}\cdots e_{i_{k_2}} \approx  w_2. 
$$

\smallskip 

Secondly, suppose that $z_1=x$ and $z_2=y$. 

Let $0\leqslant t_1\leqslant k_1$ and  $0\leqslant t_2\leqslant k_2$ be such that 
$i_{t_1}\leqslant r_1 <i_{t_1+1}$ and $j_{t_2}< n-r_2+1 \leqslant j_{t_2+1}$ (where $i_{k_1+1}=j_{k_2+1}=n$). 

As above, we have $w_1\approx x^{r_1}e_1\cdots e_{r_1}e_{i_{t_1+1}}\cdots e_{i_{k_1}}$. On the other hand, 
since $y^{r_2}\approx y^{r_2}e_{n-r_2+1}\cdots e_n$, by Lemma \ref{powers}, we have 
$$
w_2\approx y^{r_2}e_{j_1}\cdots e_{j_{k_2}}\approx y^{r_2}e_{n-r_2+1}\cdots e_n e_{j_1}\cdots e_{j_{k_2}}\approx 
y^{r_2}e_{j_1}\cdots e_{j_{t_2}}e_{n-r_2+1}\cdots e_n. 
$$

Now, in view of (\ref{xkyk}), as above $w_1\vartheta=g^{r_1}e_1\cdots e_{r_1}e_{i_{t_1+1}}\cdots e_{i_{k_1}}$ and 
$$
w_2\vartheta=y^{r_2}e_{j_1}\cdots e_{j_{k_2}}= g^{n-r_2}e_{n-r_2+1}\cdots e_n e_{j_1}\cdots e_{j_{k_2}}= 
g^{n-r_2}e_{j_1}\cdots e_{j_{t_2}}e_{n-r_2+1}\cdots e_n.
$$ 
Hence, we have 
$$
w_1\vartheta=g^{r_1}|_{\dom(w_1\vartheta)}
\quad\text{and}\quad 
\im(w_1\vartheta)=\Omega_n\setminus\{1,\ldots,r_1, i_{t_1+1}, \ldots, i_{k_1}\}
$$ 
and
$$
w_2\vartheta=g^{n-r_2}|_{\dom(w_2\vartheta)}
\quad\text{and}\quad 
\im(w_2\vartheta)=\Omega_n\setminus\{j_1,\ldots,j_{t_2}, n-r_2+1,\ldots,n\}. 
$$ 
Since $w_1\vartheta=w_2\vartheta$, then $\im(w_1\vartheta)=\im(w_2\vartheta)$ and so 
$$
\{1,\ldots,r_1, i_{t_1+1}, \ldots, i_{k_1}\}=\{j_1,\ldots,j_{t_2}, n-r_2+1,\ldots,n\}. 
$$

If $w_1\vartheta\neq\emptyset$, from $g^{r_1}|_{\dom(w_1\vartheta)}=w_1\vartheta= w_2\vartheta=g^{n-r_2}|_{\dom(w_2\vartheta)}$, 
we have $r_1=n-r_2$, by Lemma \ref{unique}, whence 
$$
\{1,\ldots,r_1, i_{t_1+1}, \ldots, i_{k_1}\}=\{j_1,\ldots,j_{t_2}, r_1+1,\ldots,n\}, 
$$
from which follows that 
$$
\{1,\ldots,r_1, i_{t_1+1}, \ldots, i_{k_1}\}=\Omega_n=\{j_1,\ldots,j_{t_2}, r_1+1,\ldots,n\}
$$
and so $\im(w_1\vartheta)=\emptyset$, i.e. $w_1\vartheta=\emptyset$, a contradiction. 
Thus $w_1\vartheta=\emptyset$.

Hence $\im(w_1\vartheta)=\emptyset=\im(w_2\vartheta)$ and so 
$$
\{1,\ldots,r_1, i_{t_1+1}, \ldots, i_{k_1}\}=\Omega_n=\{j_1,\ldots,j_{t_2}, n-r_2+1,\ldots,n\}. 
$$
Then, by Lemma \ref{zero}, we have 
$$
w_1\approx x^{r_1}e_1\cdots e_{r_1}e_{i_{t_1+1}}\cdots e_{i_{k_1}} =  x^{r_1}e_1\cdots e_n \approx e_1\cdots e_n 
\approx y^{r_2}e_1\cdots e_n = y^{r_2}e_{j_1}\cdots e_{j_{t_2}}e_{n-r_2+1}\cdots e_n \approx  w_2. 
$$

\smallskip 

Finally, we suppose that $z_1=z_2=y$. 

Let $0\leqslant t_1\leqslant k_1$ and  $0\leqslant t_2\leqslant k_2$ be such that 
$i_{t_1} < n-r_1+1 \leqslant i_{t_1+1}$ and $j_{t_2}< n-r_2+1 \leqslant j_{t_2+1}$ (where $i_{k_1+1}=j_{k_2+1}=n$). 

As above, we have $w_2\approx y^{r_2}e_{j_1}\cdots e_{j_{t_2}}e_{n-r_2+1}\cdots e_n$ and, analogously, 
$w_1\approx y^{r_1}e_{i_1}\cdots e_{i_{t_1}}e_{n-r_1+1}\cdots e_n$. 

On the other hand, as above, in view of (\ref{xkyk}), we have $w_2\vartheta=g^{n-r_2}e_{j_1}\cdots e_{j_{t_2}}e_{n-r_2+1}\cdots e_n$ and, 
similarly, we get $w_1\vartheta=g^{n-r_1}e_{i_1}\cdots e_{i_{t_1}}e_{n-r_1+1}\cdots e_n$. 
Hence, we have 
$$
w_1\vartheta=g^{n-r_1}|_{\dom(w_1\vartheta)}
\quad\text{and}\quad 
\im(w_1\vartheta)=\Omega_n\setminus\{i_1,\ldots,i_{t_1}, n-r_1+1,\ldots,n\}
$$ 
and
$$
w_2\vartheta=g^{n-r_2}|_{\dom(w_2\vartheta)}
\quad\text{and}\quad 
\im(w_2\vartheta)=\Omega_n\setminus\{j_1,\ldots,j_{t_2}, n-r_2+1,\ldots,n\}. 
$$ 
Since $w_1\vartheta=w_2\vartheta$, then $\im(w_1\vartheta)=\im(w_2\vartheta)$ and so 
$$
\{i_1,\ldots,i_{t_1}, n-r_1+1,\ldots,n\}=\{j_1,\ldots,j_{t_2}, n-r_2+1,\ldots,n\}. 
$$

If $w_1\vartheta=\emptyset$ then $\im(w_1\vartheta)=\emptyset=\im(w_2\vartheta)$, whence 
$$
\{i_1,\ldots,i_{t_1}, n-r_1+1,\ldots,n\}=\Omega_n=\{j_1,\ldots,j_{t_2}, n-r_2+1,\ldots,n\}. 
$$
and so, by Lemma \ref{zero}, we have 
$$
w_1\approx y^{r_1}e_{i_1}\cdots e_{i_{t_1}}e_{n-r_1+1}\cdots e_n 
=y^{r_1}e_1\cdots e_n \approx e_1\cdots e_n 
\approx y^{r_2}e_1\cdots e_n = 
y^{r_2}e_{j_1}\cdots e_{j_{t_2}}e_{n-r_2+1}\cdots e_n \approx  w_2. 
$$

On the other hand, if $w_1\vartheta\neq\emptyset$, from $g^{n-r_1}|_{\dom(w_1\vartheta)}=w_1\vartheta= w_2\vartheta=g^{n-r_2}|_{\dom(w_2\vartheta)}$, 
we have $n-r_1=n-r_2$, by Lemma \ref{unique}, whence $r_1=r_2$ and so 
$$
w_1\approx y^{r_1}e_{i_1}\cdots e_{i_{t_1}}e_{n-r_1+1}\cdots e_n = y^{r_2}e_{j_1}\cdots e_{j_{t_2}}e_{n-r_2+1}\cdots e_n \approx  w_2, 
$$
as required. 
\end{proof}

Next, by using Tietze transformations and applying Proposition \ref{tietze}, we deduce from the previous presentation for $\OCI_n$ 
a new one on the $n$-generators set $\{x,y,e_2,\ldots,e_{n-1}\}$ of $\OCI_n$.  
We will proceed in a similar way to what we did for $\CI_n$. 

\smallskip 

Recall that, as transformations, we have $e_1=yx$ and $e_n=xy$. Therefore, by replacing $e_1$ by $yx$ and $e_n$ by $xy$ in all relations from $U$, we obtain the following relations on the alphabet $\{x,y,e_2,\ldots,e_{n-1}\}$:
\begin{description}

\item $(U_1)$ $e_i^2=e_i$, for $2\leqslant i\leqslant n-1$; $yxyx=yx$ and $xyxy=xy$;

\item $(U_2)$ $xy=xy$ and $yx=yx$;

\item $(U_3)$ $xyx=x$ and $yxy=y$; 

\item $(U_4)$ $e_ie_j=e_je_i$, for $2\leqslant i<j\leqslant n-1$; 
    $xye_i=e_ixy$ and $yxe_i=e_iyx$, for $2\leqslant i\leqslant n-1$; $yx^2y=xy^2x$;  

\item $(U_5)$ $xe_{i+1}=e_ix$, for $2\leqslant i\leqslant n-2$; $x^2y=e_{n-1}x$ and $yx^2=xe_2$; 

\item $(U_6)$ $yxe_2\cdots e_{n-1}xy=xe_2\cdots e_{n-1}xy$.

\end{description}
 
Notice that, clearly, the relations  $xy=xy$ and $yx=yx$ are trivial and the relations $yxyx=yx$ and $xyxy=xy$ are consequences of the relation $xyx=x$. 

So, let $V$ be the following set of monoid relations on the alphabet $D=\{x,y,e_2,\ldots,e_{n-1}\}$: 
\begin{description}

\item $(V_1)$ $e_i^2=e_i$, for $2\leqslant i\leqslant n-1$; 

\item $(V_2)$ $xyx=x$ and $yxy=y$; 

\item $(V_3)$ $yx^2y=xy^2x$;  

\item $(V_4)$ $e_ie_j=e_je_i$, for $2\leqslant i<j\leqslant n-1$; 

\item $(V_5)$ $xye_i=e_ixy$ and $yxe_i=e_iyx$, for $2\leqslant i\leqslant n-1$; 

\item $(V_6)$ $xe_{i+1}=e_ix$, for $2\leqslant i\leqslant n-2$; 

\item $(V_7)$ $x^2y=e_{n-1}x$ and $yx^2=xe_2$; 

\item $(V_8)$ $yxe_2\cdots e_{n-1}xy=xe_2\cdots e_{n-1}xy$.

\end{description}
Notice that $|V|=\frac{1}{2}(n^2+3n)$. 

Thus, we have: 

\begin{theorem}\label{rankpresoci} 
The monoid $\OCI_n$ is defined by the presentation $\langle D \mid V\rangle$ on $n$ generators and $\frac{1}{2}(n^2+3n)$ relations. 
\end{theorem}

\bigskip 

\lastpage 

\end{document}